\documentclass[preprint,12pt,authoryear]{elsarticle}

\usepackage{blindtext}

\usepackage[utf8]{inputenc}
\usepackage[T1]{fontenc}
\usepackage{lmodern}
\usepackage{amsmath, amsthm, amssymb}
\usepackage{paralist}
\usepackage{float}
\usepackage{accents}
\usepackage{color}
\usepackage{bm}
\usepackage{lscape}
\usepackage{booktabs}
\usepackage{dsfont}
\usepackage{hyperref}
\usepackage{multirow}

\newcommand{\eps}{{\varepsilon}}
\renewcommand{\phi}{\varphi}
\newcommand{\R}{\mathbb{R}}
\newcommand{\Z}{\mathbb{Z}}
\newcommand{\N}{\mathbb{N}}
\newcommand{\pr}{\mathbb{P}}       
\newcommand{\ex}{\mathbb{E}}

\newcommand{\Fc}{\mathcal{F}}
\newcommand{\Ac}{\mathcal{A}}
\newcommand{\Gc}{\mathcal{G}}

\newcommand{\Xb}{\mathbf{X}}

\newcommand{\diff}{{\,\mathrm{d}}}

\newtheorem{theorem}{Theorem}
\newdefinition{remark}{Remark}
\newtheorem{assumption}{Assumption}
\newtheorem{proposition}{Proposition}
\newtheorem{example}{Example}

\begin{document}

\begin{frontmatter}

\title{A Note on Physical Dependence and Mixing Conditions for Triangular Arrays}

\author{Florian Heinrichs}
\ead{f.heinrichs@fh-aachen.de}

\affiliation{organization={Department of Medical Engineering and Technomathematics, FH Aachen - University of Applied Sciences},
            addressline={Heinrich-Mußmann-Straße 1}, 
            city={Jülich},
            postcode={52428}, 
            country={Germany}}

\begin{abstract}
	Under mild structural assumptions and regularity conditions on the marginal and conditional densities, an explicit bound on the $\beta$-mixing coefficients in terms of the physical dependence measure is provided. Consequently, weak physical dependence implies $\beta$-mixing and strong mixing for triangular arrays, complementing Hill (2025), who proved the converse implication under moment assumptions. 
\end{abstract}


\end{frontmatter}


% !TeX spellcheck = en_US
%% ====================
\section{Introduction} \label{sec:intro}
%% ====================

When moving beyond the assumption of independence, it is crucial to control the dependence structure in order to derive meaningful results, such as central limit theorems. A variety of dependence measures have been introduced, from mixing conditions \cite{rosenblatt1956,ibragimov1959,blum1963,philipp1969,kolmogorov1960,volkonskii1959} over cumulant conditions \cite{brillinger2001} to the physical dependence measure \cite{wu2005}. These measures of dependence have been generalized from (stationary) time series to triangular arrays. Most of the properties, such as covariance bounds, remain valid for triangular arrays.

The relations between different dependence measures have been extensively studied, in particular, the connections between mixing conditions, which can be summarized with the following graph \cite{doukhan1995,bradley2001,dedecker2007}.

\begin{equation*}
	\psi\text{-mixing} 
	\begin{array}{c} \Rightarrow \vspace{-3pt} \\ \not\Leftarrow \end{array}
	\phi\text{-mixing}
	\begin{array}{c} \Rightarrow \vspace{-3pt} \\ \not\Leftarrow \end{array}
	\left\{
	\begin{array}{ccc}
		\beta\text{-mixing} & \begin{array}{c} \Rightarrow \vspace{-3pt} \\ \not\Leftarrow \end{array} & \alpha\text{-mixing} \vspace{-3pt} \\
		\not\Uparrow \not\Downarrow  &  & \vspace{-3pt} \\        
		\rho\text{-mixing}  & \begin{array}{c} \Rightarrow \vspace{-3pt} \\ \not\Leftarrow \end{array} & \alpha\text{-mixing}
	\end{array}
	\right.
\end{equation*}

It is also known that $\alpha$-mixing implies the cumulant condition under certain moment bounds \citep[see Lemma 2 in][]{bucher2020}. On the other hand, the connection between mixing conditions and the physical dependence measure has been studied less. One reason might be that mixing conditions are purely distributional properties, depending only on the distribution of a time series. In contrast, weak physical dependence requires
a specific representation of the time series, and as such, is a structural property. Due
to this difference, a direct comparison of the two notions of weak dependence is
difficult. Moreover, examples of weakly dependent time series exist, that are $\alpha$-mixing but not weakly physically dependent and vice versa. These examples provide crucial insights into conditions under which one measure of dependence is controlled by another.

\cite{hill2025} recently showed that weak physical dependence follows from the strong mixing condition. In the following, we specify conditions for the other direction.

% !TeX spellcheck = en_US
%% ====================
\section{Main Result} \label{sec:method}
%% ====================

First, we recall the framework of local stationarity, as proposed by \cite{zhou2009}. Let $\eps= (\eps_i)_{i\in\Z}$ be a sequence of i.i.d. random variables and $\eps^* = (\eps_i^*)_{i\in\Z}$ an independent copy of $\eps$. Define the filtrations $\Fc_i = (\eps_k)_{k \le i}$ and $\Fc_i^* = (\dots, \eps_{-2}, \eps_{-1}, \eps_{0}^*, \eps_1, \dots, \eps_i)$.

A triangular array $\{(X_{i,n})_{i=1, \dots, n}\}_{n\in \N}$ is \textit{locally stationary}, if there exists a possibly non-linear filter $G:[0, 1] \times \R^\N \to \R$, which is continuous in its first argument, such that $X_{i, n} = G(i/n, \Fc_i)$, for all $i \in \{1, \dots, n\}$ and $n\in \N$, and $(t, \omega) \mapsto G(t, \Fc_i(\omega))$ is measurable.

For $p \ge 1$, the \textit{physical dependence measure of order $p$} of a filter $G$ with $\sup_{t\in[0, 1]}\ex[|G(t, \Fc_i)|^p] < \infty$ is defined by
\begin{equation*}
	\delta_p(G, i) = \sup_{t\in[0, 1]} \ex\big[\big| G(t, \Fc_i) - G(t, \Fc_i^*) \big|^p\big]^{1/p},
\end{equation*}
for $i\in\Z$. A triangular array $X$ is \textit{weakly physically dependent} if $\delta_p(G, i)$ vanishes, as $i\to \infty$, for some order $p \ge 1$.

For two $\sigma$-algebras $\Fc$ and $\Gc$ on a probability space, their strong mixing (or $\alpha$-mixing) coefficient is defined as
\begin{equation*}
	\alpha(\Fc, \Gc) = \sup \big\{ |\pr(A\cap B) - \pr(A)\pr(B)|: A\in\Fc, B\in\Gc \big\}.
\end{equation*}
A triangular array $X = \{(X_{i,n})_{i=1, \dots, n}\}_{n\in \N}$ is \textit{strongly mixing} (or $\alpha$-mixing), if the strong mixing coefficients
\begin{equation*}
	\alpha(k) = \sup_{n \in \N} \sup_{j=1, \dots, n-k} \alpha\big(  \sigma\{(X_{i, n})_{i=1}^j\}, \sigma\{(X_{i, n})_{i=j+k}^n\} \big)
\end{equation*}
vanish as $k$ tends to infinity. The $\beta$-mixing coefficient is defined as
\begin{equation*}
	\beta(\Fc, \Gc) = d_{\mathrm{TV}}(\pr_{\Fc , \Gc}, \pr_\Fc \otimes \pr_\Gc).
\end{equation*}
where $d_{\mathrm{TV}}$ denotes the total variation distance between probability measures, $\pr_\Fc$ and $\pr_\Gc$ denote the restrictions of $\pr$ to the respective $\sigma$-algebras, and $\pr_{\Fc , \Gc}$ is a law on the product $\sigma$-algebra defined on rectangles by $\pr_{\Fc , \Gc}(A, B) = \pr(A \cap B)$. Note that we may equivalently write
\begin{equation}\label{eq:beta_alternative}
	\beta(\Fc, \Gc) = \ex[d_{\mathrm{TV}}(\pr_{\Fc|\Gc}, \pr_\Fc)],
\end{equation}
where $\pr_{\Fc|\Gc}$ denotes the restriction of $\pr$ to $\Fc$ conditioned on $\Gc$ (see, e.\,g., the proof of Theorem 1.4 by \cite{rio2017} or \cite{dedecker2005}).
As before, a triangular array $X = \{(X_{i,n})_{i=1, \dots, n}\}_{n\in \N}$ is $\beta$-\textit{mixing}, if the $\beta$-mixing coefficients
\begin{equation*}
	\beta(k) = \sup_{n \in \N} \sup_{j=1, \dots, n-k} \beta\big(  \sigma\{(X_{i, n})_{i=1}^j\}, \sigma\{(X_{i, n})_{i=j+k}^n\} \big)
\end{equation*}
vanish as $k$ tends to infinity. It is well known that $\alpha(k) \le \tfrac{1}{2} \beta(k)$, such that $\beta$-mixing implies strong mixing \citep[see, e.\,g., Proposition 1 in Section 1.1 of][]{doukhan1995}.

As discussed before, the notions of weak physical dependence and strong mixing are fundamentally different, since the first property concerns the structure of the time series and the second only its distribution.

Trivial examples of strong mixing time series, that are not weakly physically dependent, are sequences of i.i.d. random variables without existing moments, for example, if $X_i$ is Pareto distributed. In this case, the physical dependence measure is not defined. A less trivial example is the \textit{random walk in random scenery}.

\begin{example}
	Let $(\eps_i)_{i\in\Z}$ and $(\eta_i)_{i\in\Z}$ be independent sequences of i.i.d. Rademacher random variables, so that $\pr(\eps_i \pm 1) = \pr(\eta_i \pm 1) = \tfrac{1}{2}$, for $i\in\Z$. We first define the random positions
	\[
		S_0 = 0, \quad, S_n = \sum_{i=1}^{n} \eps_i, \quad S_{-n} = \sum_{i=0}^{n-1} \eps_{-i},
	\]
	for $n\in\N$. The time series $X =(X_i)_{i\in\Z}$, defined by $X_i = \eta_{S_i}$ is $\beta$-mixing by Theorem 2.5 (ii) of \cite{denHollander2006}, and as such, strong mixing. Moreover, $X$ has no causal representation by Theorem 2.2 of \cite{denHollander2006}, as required in the definition of the physical dependence measure. 
\end{example}

The previous example shows that the physical dependence measure is not defined for arbitrary (weakly dependent) time series with existing moments. For (locally stationary) time series that have the required representation, weak physical dependence of order $p$ follows from strong mixing, whenever the time series has uniformly bounded moments of order $r > p$, by Remark 2.4 of \cite{hill2025}.

\begin{example} \label{ex:non_mixing}
	Let $(\eps_i)_{i\in\Z}$ be a sequence of independent Bernoulli distributed random variables with success probability $p \in (0, 1)$, such that $\pr(\eps_i = 1) = 1 - \pr(\eps_i = 0) = p$, for $i\in\Z$, and $\rho \in (0, \tfrac{1}{2}]$. Then, the time series $(X_i)_{i\in\Z}$ defined by
	\[
		X_i = \sum_{k=0}^{\infty} \rho^k \eps_{i-k},
	\]
	is non-strong mixing by \cite{andrews1984}. Contrarily, for the physical dependence measure of the filter $G(x) = \sum_{k=0}^\infty \rho^k x_{i-k}$ and the filtration $\Fc_i = (\eps_k)_{k\le i}$, it holds
	\[
		\delta_2(G, h) = \big(\ex[|G(\Fc_h) - \Gc(\Fc_h^*)|^2]\big)^{1/2} 
		= \big(\ex[|\rho^h (\eps_0 - \eps_0^*)|^2]\big)^{1/2} 
		\le C \rho^h,
	\]
	which vanishes, as $h\to\infty$.
\end{example}

This example shows that even in the simple case of linear processes, we should not expect strong mixing to follow from weak physical dependence. In the specific case of linear processes, a sufficient condition for strong mixing is that the marginal densities of the innovations are of bounded variation \cite[see, e.\,g. Corollary 1 in Section 2.3 of][]{doukhan1995}. In order to derive $\beta$-mixing from physical dependence, we need a similar regularity condition, as specified below.

\begin{assumption} \label{assump:loc_stat}
	Let $X = \{(X_{i,n})_{i=1, \dots, n}\}_{n\in \N}$ be a centered locally stationary time series with representation $X_{i, n} = G(i/n, \Fc_i)$. 
	$X$ is weakly physically dependent of order $1$, i.\,e., $\Theta_1 < \infty$, where $\Theta_k = \sum_{h=k}^{\infty} \delta_1(G, h)$ for $k\in \N$.
\end{assumption}

\begin{assumption} \label{assump:density}
	Let $(\Omega, \Ac, \pr)$ denote a probability space, $X = \{(X_{i,n})_{i=1, \dots, n}\}_{n\in \N}$ be a centered locally stationary time series, and $(w_m)_{m\in\N}$ be positive weights with $\sum_{m=1}^\infty w_m \le 1$. 
	\begin{enumerate}
		\item The marginal density $p_{j+k, n}$ of $(X_{j+k,n},\dots,X_{n,n})$ with respect to the Lebesgue measure has partial derivatives in $L^1(\R^{n-j-k+1})$, such that
		\begin{equation*}
			\sum_{m=1}^{n-j-k+1} w_m \|\partial_m p_{j+k, n}\|_1 \le D_1,
		\end{equation*}
		for some constant $D_1\ge 0$ and any $j \in \{1, \dots, n\}, k\in\{1, \dots, n-j\}$ and $n \in \N$.
		\item The conditional probability measure of the vector $(X_{j+k,n},\dots,X_{n,n})$ given $(X_{1,n},\dots,X_{j,n})$ has density $q_{j, k, n}$ with respect to the Lebesgue measure, such that partial derivatives are in $L^1(\R^{n-j-k+1})$ and
		\begin{equation*}
			\sum_{m=1}^{n-j-k+1} w_m \|\partial_m q_{j, k, n}\|_1 \le D_2
		\end{equation*}
		$\pr$-a.e., some constant $D_2\ge 0$ and any $j \in \{1, \dots, n\}, k\in\{1, \dots, n-j\}$ and $n \in \N$.
	\end{enumerate}
\end{assumption}

\begin{theorem}\label{thm:main}
	Let $X$ be a triangular array satisfying Assumptions \ref{assump:loc_stat} and \ref{assump:density}. There exists an absolute constant $C\ge 0$, such that for every $k\in\N$
	\begin{equation*}
		\beta(k)\le C\sqrt{D \Theta_k},
	\end{equation*}
	where $D = \max\{D_1, D_2\}$.
\end{theorem}

A direct consequence of the theorem is that strong mixing follows from weak physical dependence. This allows an application of statistical methods, that are based on the strong mixing condition, such as tests for serial correlation or outliers \citep{bucher2023,heinrichs2025}, to weakly physically dependent time series.

\begin{remark}
	Assumption \ref{assump:loc_stat} specifies that $X$ must have the required structure for the physical dependence measure to exist. It is rather mild, since most time series of practical interest possess this structure. Moreover, we do not use local stationarity in the proof of Theorem \ref{thm:main}, so that the result remains valid for arbitrary triangular arrays, as long as the physical dependence measure exists and is bounded as in the assumption.
	
	Assumption \ref{assump:density} specifies the required regularity of the marginal densities. As we have seen in Example \ref{ex:non_mixing}, the theorem is false for general innovations $\eps$.
\end{remark}

% !TeX spellcheck = en_US
%% ====================
\section{Auxiliary Results} \label{sec:aux}
%% ====================

In the following, let $\| f\|_1 = \int |f(x)| \diff x$ denote the $L^1$-norm for functions from $\R^m \to \R$ and $m\in\N$. Moreover, let $W_{1,d}$ denote the Wasserstein $1$-distance with respect to some metric $d$.

Before proving Theorem \ref{thm:main}, we first need some auxiliary results to bound the $L^1$-distance of density functions in terms of the Wasserstein distance of corresponding probability distributions.

Let $\phi\in C_c^\infty(\R)$ be a symmetric, smooth function on $\R$ with compact support, $\int\phi(x)\diff x =1, C_{\phi, 1} =\int |\phi'(x)|\diff x < \infty, C_{\phi, 2} = \int |x| \phi(x) \diff x < \infty$ and $\lim_{\eps \to 0} \tfrac{1}{\eps} \phi(x / \eps) = \delta(x)$, where $\delta$ denotes the Dirac delta. 

For $\eps > 0$, define $\eps_m =\eps w_m$ and the product mollifier on $\R^\nu$ as
\begin{equation*}
	\Phi_\eps^{(w)}(x)
	=\prod_{m=1}^\nu \frac{1}{\eps_m} \phi \big(\frac{x_m}{\eps_m}\big)
	=\prod_{m=1}^\nu \frac{1}{\eps w_m} \phi \big(\frac{x_m}{\eps w_m}\big).
\end{equation*}

For any integrable density $f$ define the mollification $f_\eps= f*\Phi_\eps^{(w)}$, where $*$ denotes the convolution. The product structure allows coordinatewise estimates. The fact $\lim_{\eps \to 0}f_\eps = f$ is a standard result in functional analysis, see, e.\,g., Chapter 5.3 of \cite{evans1998}. However, we need explicit control over the distance between $f_\eps$ and $f$, which is guaranteed by the following proposition.

\begin{proposition}[Smoothing error]\label{prop:smooth-error}
	For any density $f$ on $\R^\nu$ with partial derivatives $\partial_m f\in L^1$,
	\[
	\|f-f_\eps\|_1 \le % \sum_{m=1}^\nu \eps_m\|\partial_m f\|_1 = 
	C_{\phi, 2} \eps \sum_{m=1}^\nu w_m \|\partial_m f\|_1.
	\]
\end{proposition}

\begin{proof}
	For each $x\in\R^\nu$, it holds
	\begin{equation*}
		f(x) - f_\eps(x) = \int_{\R^\nu} \big(f(x) - f(x-y)\big) \Phi_\eps^{(w)}(y) \diff y.
	\end{equation*}
	By the gradient theorem, i.\,e., the fundamental theorem of calculus applied to the line segment from $x - y$ to $x$, it holds
	\begin{equation*}
		f(x) - f(x - y) = \int_0^1 \nabla f(x - ty) \cdot y \diff t = \sum_{m=1}^\nu \int_0^1 \partial_m f(x - ty) y_m \diff t,
	\end{equation*}
	and in particular, by Fubini's theorem and translation invariance of the Lebesgue measure,
	\begin{align*}
		\| f - f_\eps \|_1 
		& \le \int_{\R^\nu} \int_0^1 \sum_{m=1}^\nu \| \partial f(\cdot - ty)\|_1 |y_m| \diff t \Phi_\eps^{(w)} (y) \diff y \\
		& = \sum_{m=1}^\nu \| \partial_m f \|_1 \int_0^1 \int_{\R^\nu} |y_m| \Phi_\eps^{(w)} (y) \diff y \diff t.
	\end{align*}
	Now, by the definition of $\Phi_\eps^{(w)}$, we have
	\begin{equation*}
		\int_{\R^\nu} |y_m| \Phi_\eps^{(w)}(y) \diff y
		= \int_\R |x| \frac{1}{\eps w_m} \phi\big(\frac{x}{\eps w_m}\big) \diff x
		= \eps w_m \int_\R |z| \phi(z) \diff z,
	\end{equation*}
	where we used the fact, that the integral over each coordinate $y_\ell$ integrates to $1$, for $\ell \neq m$. In particular,
	\begin{align*}
		\| f - f_\eps \|_1 
		\le \eps \bigg( \int_\R |z| \phi(z) \diff z\bigg) \sum_{m=1}^\nu w_m \| \partial_m f \|_1 \le C_{\phi, 2} \eps \sum_{m=1}^\nu w_m \| \partial_m f \|_1.
	\end{align*}
\end{proof}

Next we estimate the $L^1$-difference of two mollified densities in terms of the \(W_{1,d}\)-distance of the underlying measures.

\begin{proposition}[Mollified difference in $W_{1,d}$] \label{prop:mollified_W1}
	Let $p$ and $q$ be two probability densities on $\R^\nu$ with corresponding laws $P$ and $Q$. Then for every $\eps > 0$,
	
	\begin{equation*}
		\|p_\eps - q_\eps\|_1 \le \frac{C_{\phi, 1}}{\eps}\, W_{1,d}(P,Q),
	\end{equation*}
	where $d$ is defined by
	\begin{equation*}
		d(y, z) = \sum_{m=1}^\nu w_m |y_m-z_m| 
	\end{equation*}
	with weights as in Assumption \ref{assump:density}.
\end{proposition}

\begin{proof}
	Let $\pi$ be any coupling of $P$ and $Q$. Then,
	\begin{align*}
		\| p_\eps - q_\eps \|_1
		& = \int_{\R^\nu} \iint  \Phi_\eps^{(w)}(x-y)-\Phi_\eps^{(w)}(x-z) \pi(dy,dz) dx \\
		& \le \iint \int_{\R^\nu} \big|\Phi_\eps^{(w)}(x-y)-\Phi_\eps^{(w)}(x-z)\big| dx \pi(dy,dz),
	\end{align*}
	by Fubini's theorem. As in the proof of Proposition \ref{prop:smooth-error}, we have by the gradient theorem and translation invariance of the Lebesgue measure, that 
	\begin{equation*}
		\int_{\R^\nu} \big|\Phi_\eps^{(w)}(x-y)-\Phi_\eps^{(w)}(x-z)\big| dx
		\le \sum_{m=1}^\nu \|\partial_m \Phi_\varepsilon^{(w)}\|_1 |y_m-z_m|.
	\end{equation*}
	By the definition of $\Phi_\eps^{(w)}$, 
	\begin{align*}
		\|\partial_m \Phi_\varepsilon^{(w)}\|_1
		& = \int_{\R^\nu} \frac{1}{(\eps w_m)^2} \phi'\big(\frac{y_m}{\eps w_m}\big) \prod_{\ell\neq m} \frac{1}{\eps w_\ell} \phi\big(\frac{y_\ell}{\eps w_\ell}\big) \diff y \\
		& = \int_\R \frac{1}{(\eps w_m)^2} \phi'\big(\frac{x}{\eps w_m}\big) \diff x
		= \frac{C_{\phi, 1}}{\eps w_m},
	\end{align*}
	since the integrals with respect to $y_\ell$, for $\ell \neq m$, integrate to $1$. In particular, 
	\begin{equation*}
		\int_{\R^\nu} \big|\Phi_\eps^{(w)}(x-y)-\Phi_\eps^{(w)}(x-z)\big| dx \le \frac{C_{\phi, 1}}{\eps} d(y, z).
	\end{equation*}
	Integration with respect to $\pi$ and minimizing over all couplings $\pi$ yields the claim.
\end{proof}

Combining propositions \ref{prop:smooth-error} and \ref{prop:mollified_W1} yields
\begin{equation*}
	\| p - q \|_1 \le 2C_{\phi, 2}\eps D + \frac{C_{\phi, 1}}{\eps} W_{1,d}(P,Q),
\end{equation*}
for any $\eps > 0$, where we used Assumption \ref{assump:density} to bound the weighted derivative sums by $D$. The right-hand side is minimized by $\eps = \sqrt{\tfrac{C_{\phi, 1} W_{1,d}(P,Q)}{2D C_{\phi, 2}}}$, and yields the bound
\begin{equation}\label{eq:interpolation}
	\|p-q\|_1^2 \le C D W_{1,d}(P,Q),
\end{equation}
for some constant $C \ge 0$, depending only on $\phi$.

%% ====================
\section{Proof of Theorem \ref{thm:main}} \label{sec:proof}
%% ====================

Consider some arbitrary, but fixed, $j \in \{1, \dots, n\}, k\in\{1, \dots, n-j\}$ and $n \in \N$. Let $\nu = n - j - k + 1$ and define a weighted $\ell^1$ metric on $\R^\nu$ by
\begin{equation*}
	d(x,y) = \sum_{m=1}^\nu w_m |x_m-y_m|,
\end{equation*}
for $x, y \in \R^\nu$. Moreover, let
\begin{equation*}
	\Fc_i^{(k, \ell)} = (\dots, \eps_{i-\ell - k-2}, \eps_{i-\ell-k-1}, \eps_{i-\ell-k}^*, \dots, \eps_{i-\ell}^*, \eps_{i-\ell+1}, \dots, \eps_i),
\end{equation*}
for $\ell, k\in \N_0$ and $i\in \Z$, and $\Fc_i^{(-1, \ell)} = \Fc_i$. Note that \begin{equation*}
	\Fc_{j+i}^{(\infty, i)} = (\dots, \eps_{j-1}^*, \eps_j^*, \eps_{j+1}, \dots, \eps_{j+i}), 
\end{equation*}
for any $i \in \N$. Define $\tilde{X}_{i, n} = G(i/n, \Fc_{j + i}^{(\infty, i)})$, for $i \ge k$. 
Let $\Xb = (X_{j+k,n},\dots,X_{n,n})$ and $\tilde{\Xb} = (\tilde{X}_{j+k,n},\dots,\tilde{X}_{n,n})$. 

Using a telescoping sum argument, we can rewrite
\begin{equation*}
	G(t, \Fc_{j + i}) - G(t, \Fc_{j + i}^{(\infty, i)}) = \sum_{\ell=0}^{\infty} G(t, \Fc_{j + i}^{(\ell-1, i)}) - G(t, \Fc_{j + i}^{(\ell, i)}),
\end{equation*}
since $\Fc_{j + i} = \Fc_{j + i}^{(-1, i)}$. By the triangle inequality, it holds 
\begin{align*}
	\ex[|G(t, \Fc_{j + i}) - G(t, \Fc_{j + i}^{(\infty, i)})|] 
	& \le \sum_{\ell=0}^{\infty} \ex[| G(t, \Fc_{j + i}^{(\ell-1, i)}) - G(t, \Fc_{j + i}^{(\ell, i)})|]  \\
	& \le \sum_{\ell=0}^{\infty} \delta_1(G, i + \ell) = \Theta_i,
\end{align*}
by Assumption \ref{assump:loc_stat}. In particular, $\ex[|X_{j+i,n} - \tilde{X}_{j+i,n}|] \le \Theta_i$, for any $i \ge k$. By definition of the metric $d$, it follows 
\begin{equation*}
	\ex[d(\Xb, \tilde{\Xb})] = \sum_{m=1}^\nu w_m \ex[|X_{j+m,n} - \tilde{X}_{j+m,n}|] \le \Theta_k.
\end{equation*}
Moreover, $(\Xb, \tilde{\Xb})$ is a coupling, where $\Xb$ and $\tilde{\Xb}$ have the same marginal distribution and $\tilde{\Xb}$ is independent of $\Gc_j := (X_{1, n}, \dots, X_{j, n})$. Let $P_{j+k, n}$ denote the distribution of $\Xb$ and $Q_{j, k, n}$ the conditional distribution of $\Xb$ given $\Gc_j$. By the Kantorovich-Rubinstein duality, for each realization of $Q_{j, k, n}$, we can express the Wasserstein distance as
\begin{align*}
	W_{1, d}(P_{j+k, n}, Q_{j, k, n}) 
	% & = \sup_{\|f\|_L \le 1} \bigg| \int f \diff P_{j+k, n} - \int f \diff Q \bigg| \\
	& = \sup_{\|f\|_L \le 1} \big| \ex[f(\Xb)] - \ex[f(\Xb) | \Gc_j] \big|, \quad \pr-\mathrm{a.e.}
\end{align*}
where the supremum is taken over all Lipschitz continuous functions with Lipschitz constant less than or equal to $1$. By construction, $\tilde{\Xb}$ is independent of $\Gc_j$ and $\tilde{\Xb}$ has distribution $P_{j+k, n}$, such that $\ex[f(\tilde{\Xb}) | \Gc_j] = \ex[f(\tilde{\Xb})] = \ex[f(\Xb)]$. In particular,
\begin{equation*}
	|\ex[f(\Xb)] - \ex[f(\Xb) | \Gc_j] |
	= |\ex[f(\tilde{\Xb}) | \Gc_j] - \ex[f(\Xb) | \Gc_j]|
	\le \ex\big[|f(\tilde{\Xb}) - f(\Xb)| \big| \Gc_j\big].
\end{equation*}
Taking the supremum over all Lipschitz continuous functions with Lipschitz constant bounded by $1$, yields $W_{1, d}(P_{j+k, n}, Q_{j, k, n}) \le \ex[d(\Xb, \tilde{\Xb}) | \Gc_j]$.
By taking the expectation, we finally have 
\begin{equation*}
	\ex[W_{1, d}(P_{j+k, n}, Q_{j, k, n})] \le \ex[d(\Xb, \tilde{\Xb})] \le \Theta_k.
\end{equation*}
Note that $p_{j+k,n}$ denotes the density of $P_{j+k, n}$ and $q_{j, k, n}$ the density of the conditional law $Q_{j, k, n}$. By Scheffé's theorem \citep[see, e.\,g., Lemma 2.1 in][]{tsybakov2009}, it holds $2 d_{\mathrm{TV}}(P_{j+k, n}, Q_{j, k, n}) = \|  p_{j+k,n}- q_{j, k, n} \|_1$ ($\pr$-a.e.). In particular, by Jensen's inequality and \eqref{eq:interpolation}, 
\begin{align*}
	2 \ex[ d_{\mathrm{TV}}(P_{j+k, n}, Q_{j, k, n})] % & = \ex[\|  p_{j+k,n}- q_{j, k, n} \|_1] \\
	& \le \ex[\|  p_{j+k,n}- q_{j, k, n} \|_1^2]^{1/2} \\
	& \le \sqrt{C D \ex[W_{1, d}(P_{j+k, n}, Q_{j, k, n})]}
	\le \sqrt{C D \Theta_k}.
\end{align*}
Finally, by \eqref{eq:beta_alternative}, we have
\begin{equation*}
	\beta(k) = \sup_{n \in \N} \sup_{j=1, \dots, n-k} \ex[ d_{\mathrm{TV}}(P_{j+k, n}, Q_{j, k, n})] \le \frac{\sqrt{C}}{2} \sqrt{D \Theta_k}.
\end{equation*}
\qed

\bibliographystyle{elsarticle-harv} 
\bibliography{bibliography}

\end{document}